\documentclass{article}
\usepackage[utf8]{inputenc}
\usepackage{amssymb,amsmath,amsthm}
\usepackage[english]{babel}
\usepackage{enumerate}
\usepackage[numbers]{natbib}
\usepackage[hidelinks]{hyperref}
\usepackage{mathtools}
\usepackage{subfig}
\usepackage{scalerel,stackengine}
\stackMath
\newcommand\bigghat[1]{%
\savestack{\tmpbox}{\stretchto{%
  \scaleto{%
    \scalerel*[\widthof{\ensuremath{#1}}]{\kern-.6pt\bigwedge\kern-.6pt}%
    {\rule[-\textheight/2]{1ex}{\textheight}}
  }{\textheight}%
}{0.5ex}}%
\stackon[1pt]{#1}{\tmpbox}%
}
\newtheorem{theorem}{Theorem}

\newtheorem*{definition}{Definition}

\newcommand{\prob}[0]{\textbf{$\mathbf{P}$}}
\newcommand{\xsupp}[0]{\mathbb{X}}
\newcommand{\xsam}[0]{\mathbf{X}}
\newcommand{\femp}[1]{\bigghat{F_{#1}}(x)}
\title{A Constructive Proof of the Glivenko-Cantelli Theorem}
\author{By Daniel Salnikov }
\date{October 2021}

\begin{document}
\nocite{*}
\maketitle
\section{Abstract}
The Glivenko-Cantelli theorem states that the empirical distribution function converges uniformly almost surely to the theoretical distribution for a random variable $X \in \mathbb{R}$. This is an important result because it establishes the fact that sampling does capture the dispersion measure the distribution function $F$ imposes. In essence, sampling permits one to learn and infer the behavior of $F$ by only looking at observations from $X$. The probabilities that are inferred from samples $\xsam$ will become more precise as the sample size increases and more data becomes available. Therefore, it is valid to study distributions via samples. The proof present here is constructive, meaning that the result is derived directly from the fact that the empirical distribution function converges pointwise almost surely to the theoretical distribution. The work includes a proof of this preliminary statement and attempts to motivate the intuition one gets from sampling techniques when studying the regions in which a model concentrates probability. The sets where dispersion is described with precision by the empirical distribution function will eventually cover the entire sample space.
\section{Preliminaries}
The following definitions will be used throughout the proof.
\begin{definition}{Empirical Distribution Function \cite{mood} \\}
Let $X \in \mathbb{R}$ be a random variable, and $\xsam$ be a random sample from $X$, the empirical distribution function assigns probabilities by counting how many observations are smaller than $x \in \mathbb{R}$, in essence:
\begin{equation*}
    \bigghat{F_n}(x) = \frac{1}{n} \sum\limits_{i=1}^{n} I(X_i \leq x)
\end{equation*}
\end{definition}
As a preliminary result, the work includes a proof of the following theorem:
\begin{theorem}
Let $X \in \mathbb{R}$ be a random variable, assume that $X \sim F(x)$ and that $\xsam$ is a random sample from $X$, then the empirical distribution function $\femp{n}$ converges almost surely to the theoretical distribution for each point $x \in \xsupp$. 
\begin{equation*}
    \prob_{F} \bigg( \bigghat{F_n}(x) \to F(x) \bigg) = 1.
\end{equation*}
\end{theorem}
\begin{proof}

For $x \in \xsupp$, it is possible to index a stochastic process as follows:
\begin{equation*}
    Y_i(x) = \delta_{X_i}((-\infty, x])
\end{equation*}
Therefore, $Y_i(x) \in \{0, 1\}$ and
\begin{align*}
    \prob(Y_i(x) = 1) &= \prob(X_i \leq x), \\
    &= F(x), \enspace \forall X_i.
\end{align*}
Then for each fixed $x$ the random variable $Y_i(x)$ is a Bernoulli random variable that arises from indexing the stochastic process by points in the sample space of $X$. Moreover, for all $x \in \xsupp$ the $Y_i(x)$ have finite variances and means, 
\begin{align*}
    E(Y_i(x)) &= F(x) < +\infty, \\
    Var(Y_i(x)) &= (1 - F(x))F(x) < +\infty.
\end{align*}
Thus, by The Strong Law of Large Numbers \cite{rigProb, rossMod}, for fixed $x$ the sample mean of the $Y_i(x)$ converges almost surely to the mean, 
\begin{align*}
    \prob \bigg(\overline{Y_i(x)} \to E(Y_i(x)) \bigg) &= 1, \\
    \prob\bigg(\frac{1}{n}\sum\limits_{i=1}^{n} Y_i(x) \to F(x) \bigg) &= 1, \\
    \prob \bigg( \bigghat{F_n}(x) \to F(x) \bigg) &= 1.
\end{align*}
Since the choice of $x$ is arbitrary, then for all fixed $x \in \xsupp$ it is true that
\begin{equation*}
    \bigghat{F_n}(x) \to F(x), \enspace a.s.
\end{equation*}
\end{proof}
It is worth noting that the speed of convergence depends on how heavy the tails are, kurtosis and asymmetry of $F(x)$. The more probability there is concentrated far from the center, the more the process will require larger samples to cover those regions empirically.

\section{Constructive proof}
\begin{theorem}[Glivenko-Cantelli Theorem]
Let  $X \in \mathbb{R}$ be a random variable, assume $X \sim F(x)$ and that $\xsam$ is a random sample from $X$, then as the sample size approaches infinity the empirical distribution function converges uniformly almost surely to the theoretical distribution. In essence, 
\begin{equation*}
    \sup_{x \in \mathbb{R}} \bigg \{ | \femp{n} - F(x) | \bigg \} \to 0, \enspace \textit{almost surely}.
\end{equation*}
\end{theorem}
The following proof is motivated by ideas used in the proof of Egorov's theorem in \cite{grabs}.
\begin{proof}
Assume that $X \sim F(X)$ and let $\xsam$ be a random sample from $X$. Define $\Theta$ as:
\begin{equation*}
    \Theta = \{ x\in \xsupp : \femp{n} \to F(x) \}
\end{equation*}
Then by theorem (1), 
\begin{equation*}
    \prob( \Theta ) = 1.
\end{equation*}
Now for a fixed $k \in \mathbb{N}$ and $\epsilon > 0$ define $\Theta_{k}(\epsilon)$ as:
\begin{equation*}
    \Theta_{k}(\epsilon) = \{ x \in \xsupp : | \femp{k} - F(x) | < \epsilon \}.
\end{equation*}
Thus, it is possible to verify that
$$    x \in \bigcap\limits_{k=n}^{\infty} \Theta_{k}(\epsilon)  \implies | \femp{k} - F(x) | < \epsilon, \enspace \forall k \geq n .$$
In essence, once the points fall into the $\epsilon$ ball no further empirical distribution functions in the sequence escape from the convergence radius. 
Also, for all $\epsilon > 0$ and $x \in \Theta$ there exists a $n_{\epsilon}$ such that if $k \geq n_{\epsilon}$, then 
\begin{equation*}
    | \femp{k} - F(x) | < \epsilon,
\end{equation*}
Thus, 
$$ x \in  \bigcap\limits_{k=n_{\epsilon}}^{\infty} \Theta_{k}(\epsilon) . $$
And, 
$$ \Theta \subseteq \bigcap\limits_{k=n_{\epsilon}}^{\infty} \Theta_{k}(\epsilon). $$
Furthermore, for all $x \in \xsupp$, by theorem (1), as the sample size increases, it is possible to find the required $n_{\epsilon}$ among the integers. Therefore, the union -over all integers of points where the sequence does not escape the convergence radius- will include $n_{\epsilon}$,
$$ x \in \bigcup\limits_{n=1}^{\infty} \bigcap\limits_{k=n}^{\infty} \Theta_{k}(\epsilon). $$
Hence, $x \in \Theta$ implies that $x \in \bigcup\limits_{n=1}^{\infty} \bigcap\limits_{k=n}^{\infty} \Theta_{k}(\epsilon)$, so
$\Theta \subseteq  \bigcup\limits_{n=1}^{\infty} \bigcap\limits_{k=n}^{\infty} \Theta_{k}(\epsilon) $. Now, consider the following:
\begin{align*}
    \prob (\xsupp \setminus \Theta) &= \prob(\xsupp ) - \prob(\Theta) \\
    &= 1 - 1 \\
    &= 0, \\
    \therefore \prob(\xsupp ) &= \prob(\Theta ).
\end{align*}
So that
\begin{align*}
    \prob(\xsupp) &= \prob(\Theta) \\
    &\leq \prob \bigg( \bigcup\limits_{n=1}^{\infty} \bigcap\limits_{k=n}^{\infty} \Theta_{k}(\epsilon)  \bigg) \\
    &\leq \prob( \xsupp). \\
    &\therefore  \prob \bigg( \bigcup\limits_{n=1}^{\infty} \bigcap\limits_{k=n}^{\infty} \Theta_{k}(\epsilon)  \bigg) 
    = \prob(\xsupp).
\end{align*}
And, it is possible to find $n_{\epsilon}$ to get 
\begin{align*}
    \prob(\xsupp) &= \prob(\Theta) \\
    &\leq \prob \bigg( \bigcap\limits_{k=n_{\epsilon}}^{\infty} \Theta_{k}(\epsilon)  \bigg) \\
    &\leq \prob( \xsupp). \\
    &\therefore  \prob \bigg( \bigcap\limits_{k=n_{\epsilon}}^{\infty} \Theta_{k}(\epsilon)  \bigg) 
    = \prob(\xsupp).
\end{align*}
Then, 
\begin{align*}
    \prob \bigg ( \bigcup\limits_{k = n_{\epsilon}}^{\infty} \{x \in \xsupp : | \femp{k} - F(x) | \geq \epsilon \} \bigg ) 
    &= \prob \bigg ( \xsupp \setminus  \bigcap\limits_{k=n_{\epsilon}}^{\infty} \Theta_{k}(\epsilon) \bigg) \\
    &= \prob (\xsupp) - \prob \bigg ( \bigcap\limits_{k=n_{\epsilon}}^{\infty} \Theta_{k}(\epsilon) \bigg) \\
    &= 1 - 1, \\
    \therefore \prob \bigg (  \bigcup\limits_{k = n_{\epsilon}}^{\infty} \{x \in \xsupp : | \femp{k} - F(x) | \geq \epsilon \} \bigg ) &= 0.
\end{align*}
Now fix $m$ and $\tilde{\epsilon} > 0$, and note that, given the previous limits, it is possible to find $n_m$ such that 
\begin{equation*}
    \prob \bigg ( \bigcup\limits_{k= n_m}^{\infty}  \{x \in \xsupp : | \femp{k} - F(x) | \geq \frac{1}{m} \} \bigg ) < \frac{\tilde{\epsilon}}{2^m}. 
\end{equation*}
Then, for all possible values of $m$, it is true that
\begin{align*}
    \prob \bigg ( \bigcup\limits_{m=1}^{\infty} \bigcup\limits_{k= n_m}^{\infty}  \{x \in \xsupp : | \femp{k} - F(x) | \geq \frac{1}{m} \} \bigg ) &\leq \sum\limits_{m = 1}^{\infty} \prob \bigg ( \bigcup\limits_{k= n_m}^{\infty}  \{x \in \xsupp : | \femp{k} - F(x) | \geq \frac{1}{m} \} \bigg ) \\
    &< \sum\limits_{m = 1}^{\infty} \frac{\tilde{\epsilon}}{2^{m}}\\
    &= \tilde{\epsilon} \sum\limits_{m = 1}^{\infty} \frac{1}{2^{m}} \\
    &= \tilde{\epsilon}.
\end{align*}
Since $\tilde{\epsilon}$ is arbitrary, as $\tilde{\epsilon} \to 0$, the limit results in:
\begin{equation*}
     \prob \bigg ( \bigcup\limits_{m=1}^{\infty} \bigcup\limits_{k= n_m}^{\infty}  \{x \in \xsupp : | \femp{k} - F(x) | \geq \frac{1}{m} \} \bigg ) = 0
\end{equation*}
However, if $x \notin \bigcup\limits_{m=1}^{\infty} \bigcup\limits_{k= n_m}^{\infty}  \{x \in \xsupp : | \femp{k} - F(x) | \geq \frac{1}{m} \} $, then 
for all $k \geq n_m$, $m \in \mathbb{N}$, it is true that
$$ | \femp{k} - F(x) | < \frac{1}{m} . $$
In essence, if $x \notin \bigcup\limits_{m=1}^{\infty} \bigcup\limits_{k= n_m}^{\infty}  \{x \in \xsupp : | \femp{k} - F(x) | \geq \frac{1}{m} \}  $, then $\femp{n} \to F(x)$ uniformly. And, note that $x \notin \bigcup\limits_{m=1}^{\infty} \bigcup\limits_{k= n_m}^{\infty}  \{x \in \xsupp : | \femp{k} - F(x) | \geq \frac{1}{m} \}  $ has the following probability
\begin{align*}
    \prob (\xsupp) - \prob \bigg ( \bigcup\limits_{m=1}^{\infty} \bigcup\limits_{k= n_m}^{\infty}  \{x \in \xsupp : | \femp{k} - F(x) | \geq \frac{1}{m} \}  \bigg)
    &= 1 - 0 \\
    &= 1.
\end{align*}
Therefore, the probability that $\femp{n} \to F(x)$ uniformly is equal to one, in essence, $\femp{n} \to F(x)$ uniformly almost surely.
\begin{equation*}
    \prob_{F} \bigg( \sup_{x \in \mathbb{R}} \bigg \{ | \femp{n} - F(x) | \bigg \} \to 0 \bigg) = 1.
\end{equation*}
\end{proof}
\section{Discussion}
The proof exhibits how the observed variation in sample $\xsam$ represents the dispersion that is characterized by the model $F$. Thus, as the sample size increases, the empirical process will be a more efficient representation of the model. Nonetheless, the theorem requires the use of nice intervals, and in higher dimensions classes of sets, as well as strong assumptions about independence between different observations. However, it also motivates further study of how empirical processes and nonparametric methods can be used to study the concentration of measure that characterizes the data under analysis.
\section*{Appendix}
The sampling process eventually covers all areas with positive probability. If $X$ is discrete, then 
\begin{align*}
    \prob(x_0 \notin \xsam) &= \prob(X_1 \neq x_0, \cdots, X_n \neq x_0) \\
    &= \prod\limits_{i=1}^{n} \prob(X_i \neq x_0) \\
    &= \prod\limits_{i=1}^{n}(1 - \prob(X_i = x_0)) \\
    &= (1 - f(x_0))^{n}
\end{align*}
Thus if $\prob(X = x_0) > 0$, then as $n \to +\infty$
\begin{equation*}
     \prob(x_0 \notin \xsam) \to 0.
\end{equation*}
Similarly if $I^{n} = (X^{(1)}, X^{(n)}]$ denotes the range a sample from an absolutely continuous random variable covers, then the probability that the sampling process will not cover a particular point is:
\begin{align*}
    \prob ( x_0 \notin I^{n}) &= \prob(x_0 < X^{(1)} \bigcup x_0 > X^{(n)}) \\
    &= \prob(x_0 < X^{(1)}) + \prob(x_0 >  X^{(n)}) 
\end{align*}
Note that the intersection is an absurdity ($x_0 < x^{(1)}$ and $x_0 > x^{(n)} \geq x^{(1)}$), thus, it is the empty set and it has zero probability. Moreover, as $n \to +\infty$ the sampling process will eventually sample from less concentrated regions. Suppose that $0 < F(x_0) < 1$, then
\begin{align*}
     \prob(x_0 < X^{(1)}) &= 1 - 1 + (1 - F(x_0))^n \\
     &\to 0, \enspace n \to +\infty.
\end{align*}
And, 
\begin{align*}
    \prob(x_0 >  X^{(n)})  &= (F(x_0))^{n} \\
    &\to 0, \enspace n \to +\infty.
\end{align*}
Therefore as the sample size increases the range of coverage grows to capture both the left and right tails. In essence, if $0 < F(x) < 1$, then as $n \to +\infty$, 
$$ \prob(x_0 < X^{(1)}) + \prob(x_0 >  X^{(n)}) \to 0 . $$
This motivates the study of how samples from the empirical distribution and a possible guess/initial distribution could report the regions in which $X$ appears to concentrate more probability. Suppose that $\{I_k\}_{k=1}^{m}$ is a partition of $\mathbb{R}$, then the observed clusters should resemble the concentration that a Dirichlet distribution defined over $\{\prob_{F}(I_1), \cdots, \prob_{F}(I_m)\}$ establishes for each interval $I_k$. This connection between order statistics, the empirical distribution function and the measure concentrated in each observed and/or theorized interval sets and motivates the construction of the Dirichlet process and other nonparametric techniques. Moreover, it suggests that boundaries at the tail might need specific techniques to observe such data in a sample.
\subsection*{Notation}
$\delta_{x}(A)$ is Dirac measure with respect to the set $A$ and point $x$.
\begin{equation*}
    \delta_{X_i}(A) = \begin{cases*}
                        1 \quad \textit{if } X_i \in A , \\
                        0 \quad \textit{if } X_i \notin A .  
                    \end{cases*}
\end{equation*}
\begin{itemize}
    \item $F$: Distribution function of the random variable $X$;
    \item $\prob$: Probability measure of said event;
    \item $\prob_{F}$: Probability measure of the event under model's $F$ assumption;
    \item $X^{(1)}$: Smallest order statistic;
    \item $X^{(n)}$: Largest order statistic;
    \item $\xsupp$: Sample space of the random variable $X$, all possible observations belong to this set;
    \item $\xsam$: Random sample assumed to arise from the model $F$. A finite collection of independent and identically distributed random variables that are described by the dispersion model $F$;
    \item $I(X \leq x)$: Indicator function for $X$ with respect to the set $(-\infty, x]$.
\end{itemize}

\addcontentsline{toc}{section}{References}
\bibliographystyle{unsrtnat}
\bibliography{gg_proof_DS.bib}
\end{document}